\documentclass[11pt]{amsart}

\usepackage{pxfonts,multicol}
\usepackage{amssymb}
\usepackage{eucal,mathrsfs,enumerate}

\usepackage[alphabetic]{amsrefs}
\usepackage{soul}
\usepackage[margin=1.0in]{geometry}

\newcommand{\mc}[1]{\mathcal{#1}}

\newcommand{\ov}[1]{\overline{#1}}
\newcommand{\op}[1]{\operatorname{#1}}
\newcommand{\ovop}[1]{\ov{\op{#1}}}

\newcommand{\frg}{\mathfrak{g}}
\newcommand{\ovmc}[1]{\ov{\mc{#1}}}
\newcommand\tensor{\otimes}

\newcommand{\sL}{\mathfrak{sl}}
\newtheorem{theorem}{Theorem}[section]
\newtheorem*{theorem*}{Theorem}
\newtheorem{corollary*}{Corollary}

\newtheorem{remark}[theorem]{Remark}

\newtheorem{mainconjecture*}[theorem]{Main Conjecture}

\newtheorem*{lemma*}{Lemma}

\newtheorem{claim}[theorem]{Claim}

\newtheorem*{claim*}{Claim}
\newtheorem*{approach*}{Approach}

\newtheorem*{example*}{Example}

\newtheorem{Def/Prop}{Definition/Proposition}
\newtheorem{remark/definition}{Remark/Definition}
\newtheorem{example}[theorem]{Example}
\newtheorem*{proof*}{proof}
\newtheorem*{proofsketch*}{Sketch of proof:}

\newtheorem*{prop*}{Proposition}
\newtheorem{proposition}[theorem]{Proposition}

\newtheorem{defn/thm}[theorem]{Definition/Theorem}
\newtheorem{definition/lemma}[theorem]{Definition/Lemma}
\newtheorem{definition}[theorem]{Definition}
\newtheorem{example/lemma}[theorem]{Example/Lemma}

\begin{document}

\pagenumbering{arabic}

\title[A. Gibney, and S. Mukhopadhyay]{On higher Chern classes of vector bundles of conformal blocks}
\author{A. Gibney and S. Mukhopadhyay}
\date{\today}

\begin{abstract}Here we consider higher Chern classes of vector bundles of conformal blocks on $\ovop{M}_{0,n}$, giving explicit formulas for them, and extending various results that hold for first Chern classes to them.  We use these classes to form a full dimensional subcone of the Pliant cone on $\ovop{M}_{0,n}$.\end{abstract}
\maketitle
\section{Introduction}
The moduli stack $\ovmc{M}_{g,n}$, parametrizing flat families of Deligne-Mumford stable n-pointed curves of genus $g$, carries vector bundles of conformal blocks $\mathbb{V}$ constructed using  representation theory \cite{TUY}.   When $g=0$, the bundles are globally generated \cite{Fakh}, and their Chern classes on the moduli space $\ovop{M}_{0,n}$ have positivity properties: First Chern classes are base point free, and (products of) higher Chern classes are elements of the Pliant cone.

 Here we explain how a number of vanishing results and identities governing  first Chern classes of vector bundles of conformal blocks on $\ovop{M}_{0,n}$ may be extended simply to higher Chern classes:
\begin{itemize}
\item an explicit formula for the m-th Chern class $\op{c}_m(\mathbb{V})$ on $\ovop{M}_{0,n}$ is given  in $\S$ \ref{CC}, using  an expression for the total Chern character $\op{Ch}(\mathbb{V})$ from \cite{MOPPZ}; \item additive and critical level identities governing higher Chern classes are given  in $\S$ \ref{CL}; and \item   criteria for higher Chern classes to be extremal in the nef cone are given in $\S$ \ref{extremality}.
\end{itemize}

The pseudo-effective cone $\ovop{Eff}_m(\op{X})$ is the closure of the cone generated by classes of $m$-dimensional subvarieties on a projective variety $\op{X}$.  One can define higher codimension analogues of cones of nef divisors by taking  $\op{Nef}^m(\op{X})$  to be dual to $\ovop{Eff}_m(\op{X})$.  Products of Chern classes of vector bundles of conformal blocks are elements of $\op{Nef}^m(\ovop{M}_{0,n})$. 

 Properties satisfied by effective divisors can fail 
for  cycles of codimension $m >1$  \cites{VoisinNef, DELV, Ottem, LO}.  For instance, while nef divisors are pseudoeffective,  there are varieties for which 
$\ovop{Eff}^{k}(X) \subsetneq \op{Nef}^k(X)$.  Also, while $\op{Nef}^1(X)$ is polyhedral for any Mori Dream Space $X$, there are examples of Mori Dream spaces for which $\op{Nef}^m(X)$ are not finitely generated for  $1 < m$.   On the other hand, in \cite{CLO}, it was shown that cones of  higher
codimension cycles can, in some circumstances, behave better than the cones of divisors. 

To more accurately capture properties of positive cycles of higher codimension, Fulger and Lehmann  have introduced three sub-cones of the nef cone $\op{Nef}^m(\op{X})$ \cite{LF}. The smallest of these, the Pliant cone $\op{Pl}^m(\op{X})\subset \op{Nef}^m(\op{X})$, is the closure of the cone generated by monomials in Schur classes of globally generated vector bundles on $\op{X}$.  

Products of Chern classes of vector bundles of conformal blocks  lie in the Pliant cone for $\ovop{M}_{0,n}$.  In Section \ref{FakhruddinBasis} we use vector bundles of conformal blocks to form a full dimensional subcone of  $\op{Pl}^m(\ovop{M}_{0,n})$. In Section \ref{OtherBasis}, we give subcones of  the $\op{S}_n$-invariant Pliant cone $\op{Pl}^m(\ovop{M}_{0,n})^{\op{S}_n}$.

\section{Vector bundles of conformal blocks}\label{VBS}We begin with a short description of the definition of vector bundles of conformal blocks. Original sources for the construction are \cites{Tsu, TUY}.
 The facts we use are given, primarily in the notation of \cite{Fakh}.  

 \subsection{Basic ingredients} A vector bundle of conformal blocks $\mathbb{V}(\mathfrak{g},\vec{\lambda},\ell)$ is determined by a simple Lie algebra $\mathfrak{g}$, a positive integer $\ell$, and an $n$-tuple $\vec{\lambda}=(\lambda_1,\ldots,\lambda_n)$ of dominant weights for $\mathfrak{g}$ at level $\ell$.  
Let $\mathfrak{h} \subset  \mathfrak{g}$ be a Cartan subalgebra. A dominant integral weight  $\lambda_i \in \mathfrak{h}^*$ is at level $\ell$ as long as $(\theta,\lambda_i)\leq \ell$, where $\theta\in\mathfrak{h}^*$ is the highest root of $\mathfrak{g}$, and  $( \ , \ )$ is  the Killing form, normalized so that  $(\theta,\theta)=2$.  
In \cite{BF} vector  bundles of conformal blocks were realized  as the push forward of line a bundle on the (relative) moduli stack of principal bundles on the universal curve.   Here we sketch a  construction of  fibers at smooth curves using affine Lie algebras. This is carried out in families in \cite{Sorger}.
 \subsection{Vector spaces of conformal blocks} Let $\mathfrak{g}$ be a simple Lie algebra, and let $\hat{\mathfrak{g}}:=\mathfrak{g}\otimes \mathbb{C}((t))\oplus \mathbb{C}c$, be the corresponding {\em affine Kac-Moody Lie algebra}.  Here, $c$ belongs to the center of $\hat{\mathfrak{g}}$, and the Lie bracket of 
 $\hat{\mathfrak{g}}$ is given by the following rule: 
$$[X\otimes f, Y\otimes g]=[X,Y]\otimes fg + (X,Y)\operatorname{Res}_{t=0}g\frac{df}{dt}c,$$ where $X$, $Y$ are elements of the Lie algebra $\mathfrak{g}$ and $f$, $g$ are in $\mathbb{C}((t))$.  We denote the set dominant integral weights of $\mathfrak{g}$ of level $\ell$ by $P_{\ell}(\mathfrak{g})$. 
 For each weight  $\lambda_i$, there is a unique and irreducible finite dimensional $\mathfrak{g}$-module $\op{V}_{\lambda_i}$.  
 For each $\lambda \in P_{\ell}(\mathfrak{g})$, there exists and unique irreducible, highest weight, integrable, $\hat{\mathfrak{g}}$-module $\op{H}_{\lambda}$ such that $\op{H}_{\lambda}$ is infinite dimensional,  $c$ acts on $\op{H}_{\lambda}$ by a scalar $\ell$; $V_{\lambda}\hookrightarrow \op{H}_{\lambda}$.

 Let $\op{C}$ be a smooth connected projective curve over $\mathbb{C}$, and let $U= \op{C}\setminus \{p_1,\ldots,p_n\}$. Without loss of generality, we may assume $U$ is affine (which will be true after possibly adding a point with a trivial representation, not changing the vector space described below). By $\mathfrak{g}(U)$ we mean the Lie algebra $\mathfrak{g} \otimes \mathcal{O}_{\op{C}}(U)$.  Let  $p_1$, $\ldots$, $p_n \in \op{C}$ be n smooth points and let $\lambda_1,\ldots, \lambda_n \in \mathcal{P}_{\ell}(\mathfrak{g})$.  Choose a local coordinate $\xi_i$ at each point $p_i$, and denote by $f_{p_i}$ the Laurant expansion of any element $f \in \mathcal{O}_{\op{C}}(\op{C}\setminus \{p_1,\ldots, p_n\})$.  Then for each $i$, there is a ring  homomorphism
$$\mathcal{O}_{\op{C}}(U)\to \mathbb{C}((\xi_i)), \ \ f \mapsto f_{p_i},$$ 

This defines a map (not a Lie algebra embedding)
$$\mathfrak{g}(U) \to \hat{\mathfrak{g}}_i \ \ 
\op{X}\tensor f \mapsto (\op{X}\tensor f_{p_i}, 0).$$
\noindent
 Set $\op{H}_{\vec{\lambda}} = H_{\lambda_1}\tensor \cdots H_{\lambda_n}$ and define the following by restriction of the action of $\hat{\mathfrak{g}}_i$ on $H_{\lambda_i}$:
\begin{equation}\label{action}
\mathfrak{g}(U) \times \op{H}_{\vec{\lambda}} \to \op{H}_{\vec{\lambda}} \ \ (X\tensor f, w_1\tensor \cdots w_n) \mapsto \sum_{i=1}^n w_1 \tensor \cdots w_{i-1}\tensor (X\tensor f_{p_i})\cdot w_i \tensor w_{i+1}\tensor \cdots w_n.
\end{equation}

It is possible, using that $\sum_{1\le i \le n} (X,Y)\op{Res}_{\xi_i=0}g_{p_i}df_{p_i}=0$,  to show that Equation \ref{action} defines an action of $\mathfrak{g}(U)$ on $\op{H}_{\vec{\lambda}}$. The fiber $\mathbb{V}(\mathfrak{g},\vec{\lambda},\ell)|_{(C;\vec{p})}$ is the space of coinvariants of the $\mathfrak{g}(U)$-module $\op{H}_{\vec{\lambda}}$,
which is the largest quotient of $\op{H}_{\vec{\lambda}}$ on which $\mathfrak{g}(U)$ acts trivially.

\begin{definition}
$\mathbb{V}(\mathfrak{g},\vec{\lambda},\ell)|_{(C;\vec{p})} = [\op{H}_{\vec{\lambda}}]_{\mathfrak{g}(U)}.$
\end{definition}

\section{Higher Chern classes from Chern character formula}\label{CC} Many  properties now understood about first Chern classes of vector bundles of conformal blocks were found using Macaulay 2 \cite{M2}, mainly with  the packages \cite{Schur}, and \cite{ConfBlocks}, the latter of which implements the formulas of Fakhruddin \cite{Fakh}.
We were therefore motivated to find an explicit expression for the higher Chern classes, with the hope that perhaps with the help of such software, future experimentation on higher Chern classes can be carried out.

In  \cite{MOPPZ}, a beautiful and simple formula for the total Chern character of a general Verlinde bundle on  $\ovmc{M}_{g,n}$ for all $g$, and all $n$ is given.  These bundles, over $\ovop{M}_{0,n}$ are dual to the bundles of conformal blocks.  Here we use this to find an explicit formula for $c_k(\mathbb{V})$ on $\ovop{M}_{0,n}$.

\subsection{Notation for the statement of Theorem \ref{Chern}}\label{ChernNotation}
Given a vector bundle of conformal blocks  $\mathbb{V}=\mathbb{V}(\mathfrak{g},\vec{\lambda},\ell)$, following  \cite{MOPPZ}, we set $w(\lambda)=\frac{(\lambda, \lambda+2\rho)}{2(\mathfrak{g}^*+\ell)}$,
where $\mathfrak{g}^*$ is the dual Coxeter number, and $\rho$ is  half of the sum of the positive roots.

In Theorem \ref{Chern} we express $c_m(\mathbb{V})$ as a linear combination of products of $\psi$ classes and boundary cycles $\delta_{J_1}\delta_{J_2}\cdots \delta_{J_m}$, where the $J_i$
are nested sets formed using a collection of $m$ disjoint non-empty sets $I_1$, $I_2$,$\ldots$, $I_m$ such that 
$$J_1=I_1, \ J_2=I_1\cup I_2, \ J_3=I_1\cup I_2 \cup I_3, \ \ldots, J_m=I_1\cup \cdots \cup I_m, \ \mbox{ and } 1 \in J_m^C, \  |J_m|\ge 2.$$
Then for a collection of $m$ allowable weights $\vec{\mu}=(\mu_1,\ldots, \mu_m)$, and for $j \in \{2,\ldots, m\}$, we set 
$$\mathbb{V}_{\vec{\mu}}(\lambda_{J_j})=\mathbb{V}(\mathfrak{g}, \{\lambda_i | i \in J_j\}\cup \{\mu_{j-1}^*, \mu_{j}\}, \ell),$$
while setting
$\mathbb{V}_{\vec{\mu}}(\lambda_{J_1})=\mathbb{V}(\mathfrak{g}, \{\lambda_i | i \in J_1\}\cup \{\mu_1\}, \ell), \ \mbox{ and } \ 
\mathbb{V}_{\vec{\mu}}(\lambda_{J_m^c})=\mathbb{V}(\mathfrak{g}, \{\lambda_i | i \in J_m^C\}\cup \{\mu_{m}^*\}, \ell)$.

\subsection{Statement of the formula}In  Theorem \ref{Chern}, we use the notation from Section \ref{ChernNotation}. By the splitting principal,  $[\op{Ch}(\mathbb{V})]_k=\frac{1}{k!}p_k(\mathbb{V})$, where $p_k(\mathbb{V})$ are the $k$-th power sums of the Chern roots of the vector bundle  $\mathbb{V}=\mathbb{V}(\mathfrak{g},\vec{\lambda},\ell)$,  and
\begin{equation}\label{Newton}
c_m(\mathbb{V}(\mathfrak{g},\vec{\lambda},\ell)) =  (-1)^m \sum_{\substack{(m_1,\ldots, m_j)\in \mathbb{Z}^j_{\ge 0} \\ m_1+2m_2 + \cdots + j m_j=m}}  \prod_{k=1}^j \frac{(- \op{p}_k(\mathbb{V}))^{m_k}}{m_k! k^{m_k}}.
\end{equation}
 Equation \eqref{Newton} is proved in a number of places, eg. \cite{Mead}. We give an explicit formula for the power sums  in Theorem \ref{Chern}.
 
 \begin{theorem}\label{Chern} 
\begin{equation}\label{PS}
\op{p}_k(\mathbb{V})=\sum_{m=0}^k\sum_{\vec{k}, J} \beta_{J}^{\vec{k}} \ \  \psi_1^{k_1} \psi_2^{k_2} \cdots \psi_n^{k_n} \delta_{J_1}^{k_{n+1}}\cdots \delta_{J_m}^{k_{n+m}},
\end{equation}
\begin{equation}\label{Formula}
\beta_{J}^{\vec{k}}=(-1)^{\sum_{j=1}^m k_{n+j}}{k\choose{k_1,\ldots, k_{n+m}}} \  \prod_{i=1}^n w(\lambda_i)^{k_i} \sum_{\stackrel{\vec{\mu}=(\mu_1,\ldots,\mu_m)}{ \in \mathcal{P}_{\ell}(\mathfrak{g})^m}} \prod_{1 \le j \le m+1 } \  w(\mu_j)^{k_{n+j}} \op{rk}(\mathbb{V}_{\vec{\mu}}(\lambda_{J_j})).
\end{equation}
In Eq \ref{PS} we sum over $\vec{k}=\{k_1,\ldots, k_{n+m}\}$ where
$k_1, \ldots, k_{n} \ge 0;  \mbox{ and } \  k_{n+1},\dots, k_{n+m} \ge 1,  \mbox{ such that } \ \sum_{i=1}^{n+m}k_i=k,$ and  $J=(J_1,\ldots,J_m)$. We set $J_{m+1}=J_m^C$ as described in Section \ref{ChernNotation},  and in Eq \ref{Formula},  $w(\mu_{m+1})^{k_{n+m+1}} =1$.
\end{theorem}

\begin{proof}  
Let $\mathbb{V}=\mathbb{V}(\mathfrak{g},\vec{\lambda},\ell)$.  To use  \cite{MOPPZ}, since we are working on $\ovop{M}_{0,n}$,  we note that the part of the formula with $\lambda$ contribution is $\op{exp}(0)=1$.  Moreover,  for graphs $\Gamma$ associated to stable pointed rational curves,   $h^1(\Gamma)=0$, and there are no automorphisms.  Moreover, on $\ovop{M}_{0,n}$, the Verlinde bundle is dual to $\mathbb{V}$, and we use that $p_k(\mathbb{V}^*)=(-1)^kp_k(\mathbb{V})$.  After modification of a minor typo in \cite{MOPPZ}, the main result \cite{MOPPZ} gives:
$$\op{Ch}(\mathbb{V})=\sum_{\Gamma, \vec{\mu}} (i_{\Gamma})_* \ \Big(\ \prod_{\stackrel{ \ell \in \Gamma}{legs}} \  \mbox{cont}(\ell) 
\ \prod_{\stackrel{ \nu \in \Gamma}{vertices}}  \mbox{cont}(\nu) \prod_{\stackrel{ e \in \Gamma}{edges}}  \ \mbox{cont}(e) \ \Big),$$
where one sums over all graphs $\Gamma$ dual to stable $n$-pointed curves of genus zero, and vectors $\vec{\mu}$ of attaching weights. The coefficients $\mbox{cont}(\ell)$ contributed by the legs, 
$\mbox{cont}(e)$,  by the edges, and  $\mbox{cont}(\nu)$,  by vertices of the $\Gamma$ are described below.

For the contribution of the legs, we expand the power series, given  $\Gamma$:
$$\prod_{\stackrel{\ell \in \Gamma}{legs}}\mbox{cont}(\ell) =\prod_{i=1}^n \op{exp}(w(\lambda_i) \psi_i) =  \op{exp}( \sum_{i=1}^n w(\lambda_i) \psi_i).$$

For  the contribution of the edges, we write
$$f(t)=\frac{1-e^{wt}}{t} = \frac{1-\sum_{m=0}^{\infty}\frac{(wt)^m}{m!}}{t} = - \sum_{m=0}^{\infty} \frac{w^{m+1} t^m}{ (m+1)!} = \sum_{m=0}^{\infty} \frac{(-1)^{m+1}w^{m+1}(-t)^m}{(m+1)!}.$$
Now given an edge $e \in \Gamma$, with corresponding attaching weight $\mu_e$, 
$$\mbox{cont}(e) =f(\psi_{e}'+\psi_{e}") =  \sum_{m=0}^{\infty} \frac{(-1)^{m+1}w(\mu_e)^{m+1}(-\psi_{e}'-\psi_{e}")^m}{(m+1)!}$$
and using the key identity from \cite{MOPPZ}, at the bottom of page 15, the image $i_{{\Gamma}_*}$ is
$$  \sum_{m=0}^{\infty} \frac{(-1)^{m+1}w(\mu_e)^{m+1}(\delta_e)^{m+1}}{(m+1)!}
=  \sum_{m=0}^{\infty} \frac{(-w(\mu_e) \ \delta_e)^{m+1}}{(m+1)!} = \op{exp}(-w(\mu_e) \ \delta_e)-1.$$

Expanding this out,  we  can write

\begin{multline}
\op{Ch}(\mathbb{V})=\sum_{m\ge0} \sum_{\stackrel{\Gamma \mbox{\tiny{ with m edges}}}{ \vec{\mu}=\{\mu_{e_1},\ldots,\mu_{e_m}\}}}
 \ \prod_{\stackrel{ \nu \in \Gamma}{vertices}}  \mbox{cont}(\nu) \op{exp}(\sum_{i=1}^n w(\lambda_i) \psi_i )\prod_{j=1}^m\big(\op{exp}(-w(\mu_{e_j}) \ \delta_{e_j})-1\big) \\
 =\sum_{m\ge 0}
 \sum_{\stackrel{\Gamma \mbox{\tiny{ with m edges}}}{ \vec{\mu}=\{\mu_{e_1},\ldots,\mu_{e_m}\}}}
 \ \prod_{\stackrel{ \nu \in \Gamma}{vertices}}  \mbox{cont}(\nu) \sum_{k \ge m} \frac{1}{k!} \sum_{\vec{k}=\{k_1,\ldots,k_{n+m}\}}
  {k\choose{k_1,\ldots,k_{n+m}}}
 \prod_{i=1}^n (w(\lambda_i) \ \psi_i)^{k_i}  \prod_{j=1}^{m} (-w(\mu_{e_j}) \ \delta_{e_j})^{k_{n+j}}.
 \end{multline}
 Above we sum over $\vec{k}=\{k_1,\ldots,k_{n+m}\}$, such that  
 $$k_1, \ldots, k_{n} \ge 0; \ \ k_{n+1},\dots, k_{n+m} \ge 1, \ \mbox{and } \sum_{i=1}^{n+m}k_i=k,$$
 and
  $\vec{\mu}=\{\mu_{e_1},\ldots,\mu_{e_m}\} \in \mathcal{P}_{\ell}(\mathfrak{g})^m$.

We write $J_i = I_1 \cup \cdots \cup I_i$, where the sets $I_1$, $I_2$, $\ldots$, $I_m$, and $I_{m+1}=(I_1\ldots I_m)^c$ form a partition of $[n]$, with $I_j\ne \emptyset$ for all $j$. For $(C,p_1,\ldots, p_n) \in \Delta_{J_1} \cap \cdots \cap \Delta_{J_m}$, and $\Gamma(C)$ has vertices $v_i$ corresponding to each smooth component $C_i$ of $C$, then we can label the vertices so that
\begin{itemize}
\item $v_1$ is connected to $v_2$ by the edge $e_1$ labeled with an attaching weight $\mu_1$;
\item for $i\in \{2,\ldots, m\}$, the vertex $v_i$ is connected by and edge to $v_{i-1}$ labeled with $\mu_{i-1}^*$, and connected by an edge labeled with $\mu_{i+1}$ to the vertex  $v_{i+1}$;
\item the last vertex $v_{m+1}$ is connected to $v_{m}$ by an edge labeled with an attaching weight $\mu_{m}^*$;
\item each vertex $v_i$ has $|I_i|$ adjacent half edges labeled with the given weights $\lambda(I_i)=\{\lambda_j : j \in I_i\}$
\end{itemize}

We set 
$$\mathbb{V}_{\vec{\mu}}(\lambda_{I_i})=\mathbb{V}(\mathfrak{g}, \{\lambda_j : j \in I_i\}\cup \{\mu_{i-1}^*,\mu_{i}\}, \ell)),$$
where $\vec{\mu}=\{\mu_{i-1}^*,\mu_i\}$, and if $i=1$, then $\mu_{i-1}^*=0$, and if $i=m+1$, then $\mu_{i}=0$.

When we solve for the degree one part $[\op{Ch}(\mathbb{V})]_1$, we are concerned with summing over trees having either zero or one edges. 
If there is just one vertex $v$ and no edges, then there is no need for attaching weights and:
$$cont(v)=\op{rank}(\mathbb{V})=\op{rank}(\mathbb{V}(\mathfrak{g}, \{\lambda_1,\ldots, \lambda_n\}, \ell)).$$

If there is just one edge $e$, with attaching weight $\mu_{e}$, then we work with partitions of the marked points into two sets $I$ and $I^C$, each having at least two elements.   One obtains:
\begin{equation}\label{First}[\op{Ch}(\mathbb{V})]_1=
  \op{rk}(\mathbb{V})  \sum_{i=1}^n w(\lambda_i)\psi_i
 -  \sum_{\stackrel{I \subset [n], 1 \in I^C}{ \mu_e \in \mathcal{P}_{\ell}(\mathfrak{g})}}\ w(\mu_e) \  \op{rank}(\mathbb{V}_{\mu_e}(\lambda_{I})) \  \op{rank}(\mathbb{V}_{\mu^*_e}(\lambda_{I^C}))   \ \delta_I.
 \end{equation}
Equation \ref{First} is therefore the same as Fakhruddin's formula \cite{Fakh}, as rewritten in \cite{Muk3}.

In general, for  the degree k part $[\op{Ch}(\mathbb{V})]_k$, we sum over trees having $m \le k$ edges.  By doing this, one obtains 
$$[\op{Ch}(\mathbb{V})]_k= \sum_{m=0}^k\sum_{\vec{k}, J} \alpha_{J}^{\vec{k}} \ \  \psi_1^{k_1} \psi_2^{k_2} \cdots \psi_n^{k_n} \delta_{J_1}^{k_{n+1}}\cdots \delta_{J_m}^{k_{n+m}},$$
where we sum over $\vec{k}=\{k_1,\ldots, k_{n+m}\}$ where
$$k_1, \ldots, k_{n} \ge 0;  \mbox{ and } \  k_{n+1},\dots, k_{n+m} \ge 1,  \mbox{ such that } \ \sum_{i=1}^{n+m}k_i=k,$$ and 
$J=(J_1,\ldots,J_m)$,  as above, and where
\begin{equation}\label{Last}
\alpha_{J}^{\vec{k}}=\frac{(-1)^{{\sum_{j=1}^m k_{n+j}}}}{k!}{k\choose{k_1,\ldots, k_{n+m}}} \sum_{\stackrel{\vec{\mu}=(\mu_1,\ldots,\mu_m\})} {\in \mathcal{P}_{\ell}(\mathfrak{g})^m}} \prod_{i=1}^n w(\lambda_i)^{k_i}
\prod_{j=1}^{m+1}w(\mu_j)^{k_{n+j}} \op{rk}(\mathbb{V}(\mathfrak{g},\lambda(I_j) \cup \{\mu_{j-1}^*,\mu_j\},\ell )).
\end{equation}
Here again, in Eq \ref{Last}, we assume things are labeled so that if $j=1$, then $\mu_{j-1}^*=0$, while if $j=m+1$, we set $w(\mu_{m+1})^{k_{n+m+1}}=1$. 
Since $[\op{Ch}(\mathbb{V})]_k=\frac{1}{k!}p_k(\mathbb{V})$, multiplying  by $k!$ to obtain $\beta_{I}^{\vec{k}}$, we are finished.
\end{proof}

\begin{remark} Assuming a mild generalization of  \cite{EV}, Corollary B.3, (see also \cite{Fakh}, Theorem 3.2), in principal one can derive a formula for the Newton classes $p_k(\mathbb{V})$. We have checked that the coefficients so obtained agree with the statement given in Theorem \ref{Chern}.

\end{remark}

\section{Additive and critical level identities for higher Chern classes}\label{CL}
A number vanishing statements and relations have been found for first Chern classes of vector bundles of conformal blocks. Here we give higher Chern class analogues of a number of them.
\subsection{Additive identities}   In \cite{BGMB}, Proposition $4.1$, it was proved that  if certain ranks are satisfied, additive relations hold for first Chern classes. The following is an additive identity for higher Chern classes.

\begin{proposition} \label{Additive} For $\vec{\nu}_1\in \mc{P}_{\ell_1}(\mathfrak{g})$, $\vec{\mu}_1\in \mc{P}_{m_1}(\mathfrak{g})$, and $\vec{\nu}_1+\vec{\mu}_1\in \mc{P}_{\ell_1+m_1}(\mathfrak{g})$, suppose:
  $$\op{rk}(\mathbb{V}(\mathfrak{g}, \vec{\nu}_1,\ell_1))=1, \mbox{ and } \op{rk}(\mathbb{V}(\mathfrak{g}, \vec{\mu}_1, m_1))=\op{rk}(\mathbb{V}(\mathfrak{g}, \vec{\nu}_1+ \vec{\mu}_1,\ell_1+m_1))=\delta.$$
   Then
$$\op{c}_m(\mathbb{V}(\mathfrak{g}, \vec{\nu}_1+ \vec{\mu}_1,\ell_1+m_1))=\sum_{k=0}^m{m+\delta-k\choose k} \   c_1(\mathbb{V}(\mathfrak{g}, \vec{\nu}_1,\ell_1))^k \ c_{m-k}(\mathbb{V}(\mathfrak{g}, \vec{\mu}_1,m_1)).
$$
\end{proposition}

\begin{proof}This follows  from the fact that Chern characters are multiplicative for tensor products,  
and by \cite{BGMB}, Proposition 2.1, which gives that
 $\mathbb{V}(\mathfrak{g}, \vec{\nu}_1+ \vec{\mu}_1,\ell_1+m_1) \cong \mathbb{V}(\mathfrak{g}, \vec{\nu}_1,\ell_1)\otimes \mathbb{V}(\mathfrak{g}, \vec{\mu}_1,m_1).$
\end{proof}

\begin{example}\label{Simple}
On $\ovop{M}_{0,7}$, the   bundle $\mathbb{V}_1=\mathbb{V}(\sL(3), \{(2\omega_1)^2, (\omega_1)^5\}, 2)$ has rank 3, $\mathbb{V}_2=\mathbb{V}(\sL(3), \{(\omega_1)^6, 0\}, 1)$ has rank 1, and $\mathbb{V}_3=\mathbb{V}(\sL(3), \{(3\omega_1)^2, (2\omega_1)^4, \omega_1\}, 3)$ has rank 3.
Using Proposition \ref{Additive}, 
$c_2(\mathbb{V}_3)=c_2(\mathbb{V}_1)+3 c_1(\mathbb{V}_1)c_1(\mathbb{V}_2) + c_1(\mathbb{V}_2)^2$.
In fact, this simplifies further, as we'll see later in Example \ref{Later}.
\end{example}

\subsection{Critical level identities}
All first Chern classes of vector bundles for $\sL(r+1)$ are zero if $\ell$ is ``above the critical level" (Def \ref{CriticalLevelDef}). This was proved in \cite{Fakh} for $r=1$, and in \cite{BGMA}, Theorem 1.3,  for general $r$. Bundles in type $\op{A}$ 
 at the critical level have so-called partner bundles, and in \cite{BGMA}, Proposition 1.6, it was shown that first Chern classes of partner bundles are equal.  
 We have shown the $k$-th Chern class of any bundle at the critical level can be written as a sum of products of  Chern classes of its critical level partner.  By \cite{BGMA}, Proposition 1.6, the sum of the ranks of the critical level partner bundles equal the rank of the bundle of coinvariants.  By our analogue for higher Chern classes, if $\mathbb{V}$ is a critical level bundle, and its partner has rank one, $c_k(\mathbb{V})$ is a product  of first Chern classes. 

\begin{definition}\label{CriticalLevelDef}(\cite{BGMA}) Let $\vec{\lambda}=(\lambda_1,\dots,\lambda_n)$ be an $n$-tuple of  {\em normalized} integral weights for $\sL({r+1})$, assume that $r+1$ divides $\sum_{i=1}^n|\lambda_i|$, and define the critical level for the pair  $(\sL({r+1}),\vec{\lambda})$ to be
$$ \op{CL}(\sL({r+1}),\vec{\lambda})=-1+\frac{1}{r+1}\sum_{i=1}^n|\lambda_i|.$$
 A vector bundle $\mathbb{V}(\mathfrak{g},\vec{\lambda},\ell)({\sL({r+1}),\ell})$  is said to be a critical level bundle
 if $\ell= \op{CL}(\sL({r+1}),\vec{\lambda})$ and $\vec{\lambda} \in P_{\ell}(\sL({r+1}))^n$.  We say that $\mathbb{V}(\sL({r+1}),\vec{\lambda},\ell)$ is above the critical level if $\ell > \op{CL}(\sL({r+1}),\vec{\lambda})$, and below the critical level if $\ell < \op{CL}(\sL({r+1}),\vec{\lambda})$.
\end{definition}

By \cite{BGMA}, Proposition 1.3, if $\ell >  \op{CL}(\sL({r+1}),\vec{\lambda})$, then $\mathbb{V}(\sL({r+1}),\vec{\lambda},\ell)$ is a constant bundle and so all of its Chern classes are trivial.  
On the other hand, if  $\ell =  \op{CL}(\sL({r+1}),\vec{\lambda})$, then $\mathbb{V}(\sL({r+1}),\vec{\lambda},\ell)$ has a partner bundle $\mathbb{V}(\sL({\ell+1}),\vec{\lambda}^T,r)$
where  the weight $\lambda_i^T$ is obtained by taking the transpose of the Young diagram associated to the weight $\lambda
  _i$.  Moreover, by \cite{BGMA}, Proposition 1.6 (b): $$c_1(\mathbb{V}(\sL({r+1}),\vec{\lambda},\ell))=c_1(\mathbb{V}(\sL({\ell+1}),\vec{\lambda}^T,r) ).$$

This was proved by Fakhruddin  for $r=1$ \cite{Fakh}.

The following is a critical level identity for higher Chern classes:

\begin{proposition}\label{CritProp}If $\mathbb{V}(\sL({r+1}),\vec{\lambda},\ell)$ and $\mathbb{V}(\sL({\ell+1}),\vec{\lambda}^T, r)$ are critical level partner bundles, then
\begin{multline}
c_k(\mathbb{V}(\sL({r+1}),\vec{\lambda},\ell))=\\
\sum_{\stackrel{1\le n_1, \cdots,n_j \le k}{\sum_{i=1}^j i \cdot n_i=k}}(-1)^{k-(n_1+\dots+n_j)}\binom {\sum_{i=1}^j n_i} {n_1,\cdots, n_j}c_{1}(\mathbb{V}(\sL({\ell+1}),\vec{\lambda}^T, r))^{n_1}\cdots c_{j}(\mathbb{V}(\sL({\ell+1}),\vec{\lambda}^T, r))^{n_j}.
\end{multline}
\end{proposition}

\begin{proof}
In \cite{BGMA} it was shown that one has the short exact sequence
$$0 \longrightarrow \mathbb{V}(\sL({r+1}),\vec{\lambda},\ell)^* \hookrightarrow \mathbb{A}^*_{\sL({r+1}),\vec{\lambda}} 
 \twoheadrightarrow \mathbb{V}(\sL({\ell+1}),\vec{\lambda}^T, r) \longrightarrow 0.$$

The bundle $\mathbb{A}^*_{\sL({r+1}),\vec{\lambda}}$ is constant, and so it's Chern classes are all zero. Now the result follows from the fact that Chern polynomials are multiplicative in short exact sequences. 

\end{proof}

\begin{example}\label{Later}
In Example \ref{Simple} we saw by Proposition \ref{Additive}, 
$c_2(\mathbb{V}_3)=c_2(\mathbb{V}_1)+3 c_1(\mathbb{V}_1)c_1(\mathbb{V}_2) + c_1(\mathbb{V}_2)^2$.
Since $\mathbb{V}_1$ is at the critical level, we can write this as $c_2(\mathbb{V}_3)=c_1(\mathbb{V}_1)^2+3 c_1(\mathbb{V}_1)c_1(\mathbb{V}_2) + c_1(\mathbb{V}_2)^2$.
\end{example}

\begin{example}\label{SL2Identity}
For any critical level bundle $\mathbb{V}(\sL(2), \vec{\lambda},\ell)$,
$$c_k(\mathbb{V}(\sL(2), \vec{\lambda},\ell))=c_1(\mathbb{V}(\sL({\ell+1}),\vec{\lambda}^T,1))^k=c_1(\mathbb{V}(\sL(2), \vec{\lambda},\ell))^k.$$

It's CL partner is $\mathbb{V}(\sL({\ell+1}),\vec{\lambda}^T,1)$, and any level one bundle in type $\op{A}$  has rank one by \cite{Fakh}.  
\end{example}

\begin{example}\label{Kazanova} 
In \cite{Kaz}, Kazanova proved that all $\op{S}_n$-invariant bundles for $\sL(n)$ on $\ovop{M}_{0,n}$ of rank one are of the form $\mathbb{V}(\sL({n}), \lambda^n,\ell)$, where $\lambda=(\ell-m)\omega_i+m\omega_{i+1}$.  For the bundle to be at the critical level we must also have $i=1$ and $m=1$.  The dual partner is of the form $\mathbb{V}(\sL({r+1}), (\omega_1+\omega_r)^n,n-1)$, and by \cite{BGMA} the partner will have rank equal to the rank of the bundle of coinvariants minus one.  One has:
$$c_k(\mathbb{V}(\sL({r+1}), (\omega_1+\omega_r)^n,n-1))=c_1(\mathbb{V}(\sL({r+1}), (\omega_1+\omega_r)^n,n-1))^k=c_1(\mathbb{V}(\sL({n}), ((r-1)\omega_1+\omega_2)^n, r))^k.$$
She also proved that one can express such divisors as sums of first Chern classes of level one bundles:
$$c_1(\mathbb{V}(\sL({n}), ((r-1)\omega_1+\omega_2)^n,r))=(r-1)c_1(\mathbb{V}(\sL({n}), (\omega_1)^n,1))+c_1(\mathbb{V}(\sL({n}), (\omega_2)^n,1))=c_1(\mathbb{V}((\sL({n}), \omega_2^n,1)),$$
since $c_1(\mathbb{V}(\sL({n}), (\omega_1)^n,1))$ is above the critical level, and hence is zero.  So
$$c_k(\mathbb{V}(\sL({r+1}), (\omega_1+\omega_r)^n,n-1))=c_1(\mathbb{V}((\sL({n}), \omega_2^n,1))^k, \ \ \mbox{ for all } r.$$
In particular,  this gives an infinite family of $\op{S}_n$ -invariant bundles of rank equal to codimension one in its coinvariants, whose higher Chern classes are all powers of first Chern classes of the same bundle of level one. 

\end{example}

\begin{example}Consider  the rank two bundle $\mathbb{V}=\mathbb{V}(\sL(2), (\omega_1^4,2\omega_1),2)$ on the smooth surface $\ovop{M}_{0,5}$. Since $4|\omega_1|+|2\omega_1|=6=2(2+1)$, we have that $\mathbb{V}$ is at the critical level.  It's critical level partner is the bundle $\mathbb{V}(\sL(3), (\omega_1^4,\omega_2),1)$, which has rank one. By Proposition 1.6(b) in \cite{BGMA}, one has  $c_1(\mathbb{V}(\sL(2), (\omega_1^4,2\omega_1),2)=c_1(\mathbb{V}(\sL(3), (\omega_1^4,2\omega_2),1))$, and hence by Proposition \ref{CritProp}: \begin{equation}\label{identity}c_2(\mathbb{V}(\sL(2), (\omega_1^4,2\omega_1),2)=(c_1(\mathbb{V}(\sL(2), (\omega_1^4,2\omega_1),2))^2.
\end{equation} We  verify Equation \eqref{identity} numerically first by computing the $c_2(\mathbb{V}(\sL(2), (\omega_1^4,2\omega_1),2)$ using the explicit formula in Theorem \ref{Chern}, and then computing the right hand side using the first Chern class formula in \cites{Fakh, Muk3}.

First we note  the set of level-$2$ weights for $\sL(2)$ is given by $\mathcal{P}_{2}(\sL(2))=\{0,\omega_1,2\omega_1\}$ and Casimir scalars are $w(0)=0$, $w(\omega_1)=\frac{3}{16}$ and $w(2\omega_1)=\frac{1}{2}$. Applying the first Chern class formula for conformal blocks in \cites{Fakh, Muk3}, we get 
$$c_1(\mathbb{V}(\sL(2), (\omega_1^4, 2\omega_1),2))=2\big(\frac{3}{16}\sum_{i=1}^4\psi_i+\frac{1}{2}\psi_5\big)-\frac{3}{8}\sum_{i=1}^4\delta_{i,5}
-\frac{1}{2}\sum_{ij \in \{1234\}}\delta_{i,j}.$$
Using intersection theory, one can see that $c_1(\mathbb{V}(\sL(2),(\omega_1^4,2\omega_1),2))^2$ is one. 

Now we know that $p_2(\mathbb{V})=c_1(\mathbb{V})^2-2c_2(\mathbb{V})$ for any vector bundle $\mathbb{V}$. Thus we will be done if we can show that $p_2(\mathbb{V}(\sL(2),(\omega_1^4,2\omega_1),2))=-1$.   Using Theorem \ref{Chern}:
\begin{multline}
16^2 \ p_2(\mathbb{V}(\sL(2),(\omega_1^4,2\omega_1),2))
=2\big(9 \sum_{i=1}^4 \psi_i^2 + 64 \psi_5^2\big) +36 \sum_{ij \in \{1234\}, i\ne j} \psi_i \cdot \psi_j +96 \psi_5\sum_{i=1}^4\psi_i \\
-24 \cdot \big( 8\sum_{k\in \{1,5\}, ij \in \{1,5\}^c}\psi_k \cdot \delta_{i,j}+6\sum_{\{k,m\}= \{1,5\}, i \in \{k,m\}^c}\psi_k \cdot \delta_{i,m}\big)
-2 \cdot 64 \psi_5 \cdot \sum_{ij\in \{5\}^c}\delta_{ij}\\
+64\sum_{ij \in \{5\}^c}\delta^2_{ij}+18 \sum_{i=1}^4\delta_{i,5}^2
+48\big( \sum_{i,j \in \{1,5\}^c}\delta_{1,i}\cdot \delta_{j,5}+\sum_{ij\in \{1,5\}^c}\delta_{1,5}\cdot \delta_{i,j}+\sum_{\{ijk\}=\{15\}^c}\delta_{i,j}.\delta_{k,5}\big)\\
=2(36+64)+36\times 12 +96\times 8-24\times (24+18) -128\times 6-64\times 6 -18\times 4+48(2+2+2+3+3)
=-16^2
\end{multline}
This completes the verification.
\end{example}

\section{Criteria for extremality}\label{extremality}

\subsection{Extremality from the critical level}
\begin{definition}\label{5}
Let $\cup_{1\le j \le k+3} \ J_j=N$ be a partion of $N=\{1,\dots, n\}$ into $k+3$ nonempty sets $J_j$. Let  $Z_{\vec{J}}$ be the image  in $\overline{\operatorname{M}}_{0,n}$ of 
the clutching map 
$\overline{\operatorname{M}}_{0,k+3} \twoheadrightarrow Z_{\vec{J}} \hookrightarrow \overline{\operatorname{M}}_{0,n}$
defined by sending a point $X=(C,p_1,\ldots,p_{k+3}) \in \overline{\operatorname{M}}_{0,k+3}$  to a point in $\overline{\operatorname{M}}_{0,n}$, given by attaching $k+3$-points $(\mathbb{P}^1, J_j\cup \{q_j\})$ to $X$
by identifying $p_j$ with $q_j$.   
\end{definition}

\begin{proposition}\label{empty}
Suppose $r\geq 1$ and $\ell\geq 1$ and let $\vec{\lambda}=(\lambda_1,\dots,\lambda_n)$ be an $n$-tuple in $P_{\ell}(\sL(r+1))$. Let $J_1,\dots, J_{k+3}$ be any partition of $N$ into $k+3$ non-empty sets. Without loss of generality assume that $\lambda(J_i)=\sum_{a\in J_i} |\lambda_i|$ are ordered, i.e. $\lambda(J_1)\leq \dots \leq \lambda(J_{k+3})$. If $\sum_{i=1}^{k+2}\lambda(J_i) \leq \ell+r$, then $c_k(\mathbb{V}(\sL(r+1),\vec{\lambda},\ell)$ (possibly trivial) contracts the $k$-dimensional $F$-cycle $Z_{\vec{J}}$. In particular, $c_k(\mathbb{V}(\sL(r+1),\vec{\lambda},\ell))$ is extremal in $\operatorname{Nef}^k(\ovop{M}_{0,n})$.
\end{proposition}
To prove Proposition \ref{empty} we use the {factorization theorem} in \cite{TUY}  and vanishing above the critical level. We now recall the statement:

\begin{theorem}\cite{TUY} \label{Factorization} Let $(C_0; p_1,\ldots, p_n)$ be a stable $n$-pointed curve of genus $0$ where $C_0$ has a node $x_0$. 

Let $\nu: C_1 \sqcup C_2 \to C_0$ the normalization of $C_0$ at $x_0$ and $\nu^{-1}(x_0)=\{x_{1},x_2\}$, with $x_i \in C_i$, then the fiber $\mathbb{V}(\mathfrak{g}, \vec{\lambda}, \ell)|_{(C_0; \vec{p})}$ is isomorphic to $$\bigoplus_{\mu \in \mathcal{P}_{\ell}(\mathfrak{g})} \mathbb{V}(\mathfrak{g}, \lambda(C_1) \cup \{\mu\}, \ell)|_{(C_1; \{p_i \in C_1\}\cup \{x_1\})}\tensor \mathbb{V}(\mathfrak{g}, \lambda(C_2)\cup \{\mu^{*}\}, \ell)|_{(C_2; \{p_i \in C_2\}\cup \{x_2\})},
$$
where $\lambda(C_i)=\{\lambda_j | p_j \in C_i\}$.
\end{theorem}

\begin{proof}(of Proposition \ref{empty})
The $k$-dimensional $F$-cycle $Z_{\vec{J}}$ lies in a boundary component of $\ovop{M}_{0,n}$ that  is isomorphic to the $k$-cycle in $\ovop{M}_{0, |J_1\cup\dots\cup J_{k+2}|+1}\times \{pt\}$ under the attaching map described in Definition \ref{5}. To show  $c_k(\mathbb{V}(\sL(r+1),\vec{\lambda},\ell))$ contracts $Z_{\vec{J}}$ we can use Factorization to examine $\mathbb{V}(\sL(r+1),\vec{\lambda},\ell)$ at points
that lie in the boundary component.  By factorization, at points in $\ovop{M}_{0, |J_1\cup\dots\cup J_{k+2}|+1}\times \ovop{M}_{0,|J_{k+3}|+1}$, the conformal block bundle decomposes as a direct sum, where each factor is of the form $$\mathbb{V}(\sL(r+1), \{\lambda_{i}\}_{i\in J}\cup \mu,\ell)  \otimes \mathbb{V}(\sL(r+1), \{\lambda_{i}\}_{ i\in J_{k+3}}\cup \mu^*,\ell),$$ such that $\mu \in P_{\ell}(\sL(r+1))$ and $J=J_1\cup\dots\cup J_{k+2}$. We note that $|\mu|$ is bounded by $r \cdot \ell$. Hence $\sum_{i \in J}|\lambda_i| + |\mu| \leq \ell +r + r\cdot\ell< (\ell+1)(r+1)$. This implies that $\ell$ is above the critical level for $\mathbb{V}(\sL(r+1), \{\lambda_{i}\}_{ i\in J}\cup \mu,\ell)$. In \cite{BGMA}, it
was  shown that in case $\ell$ is above the critical level, then the bundle itself is trivial.   In particular, all Chern classes will be trivial.
It follows  that the $k$-th Chern class of $\mathbb{V}(\sL(r+1),\vec{\lambda},\ell)$ pulled back to $\ovop{M}_{0, |J_1\cup\dots\cup J_{k+2}|+1}\times \operatorname{pt}$ is trivial. 

\end{proof}

\begin{example}
 On $\ovop{M}_{0,2m}$, the bundle $\mathbb{V}(\sL(2),\omega_1^{2m},m-1)$ has critical level partner  $\mathbb{V}(\sL(m),\omega_1^{2m},1)$, a bundle of rank one. By Section \ref{CL},  
$$c_k(\mathbb{V}(\sL(2),\omega_1^{2m},m-1)=c_1(\mathbb{V}(\sL(2),\omega_1^{2m},m-1)^k.$$  Let $J_i=\{i\}$ for $1\leq i\leq k+2$ and $J_{k+3}=\{1,\dots,2m\}\backslash \cup_{i=1}^{k+2}J_i$. The sets $\vec{J}=(J_1,\dots,J_{k+3})$ define a $k$-dimensional $F$-cycle under the  attaching map $\ovop{M}_{0,k+3}\rightarrow \ovop{M}_{0,2m}$. 

By Prop \ref{empty}, to show the pull back of $\mathbb{V}(\sL(2),\omega_1^{2m},m-1)$ under the ataching map is trivial, it is enough to guarantee  the critical level of the bundle $\mathbb{V}(\sL(2),\omega_1^{k+2},\mu;m-1)$ on $\ovop{M}_{0,k+3}$ is less than $m-1$. This is true for example when $k+2\leq m$.  More generally, if $k+2\leq m$, then $c_k(\mathbb{V}(\sL(2),\omega_1^{2m},m-1)$ are all extremal in $\operatorname{Nef}^k(\ovop{M}_{0,2m})$.
\end{example}

\subsection{Extremality from the theta level}
 The theta level (Def \ref{FSVCriticalLevel}), comes from the interpretation of a vector space of conformal blocks as an explicit quotient \cite{Beauville}, Proposition 4.1, see also \cite{FSV2}), and holds in all types. It was also shown in \cite{BGMA}, that conformal blocks divisors above the theta level are trivial.

  \begin{definition}\label{FSVCriticalLevel} \cite{BGMA} Given a pair $(\mathfrak{g}, \vec{\lambda})$, one refers to
$$\theta(\frg,\vec{\lambda})= -1+ \frac{1}{2} \sum_{i=1}^n\lambda_i
(H_{\theta})\in \frac{1}{2}\mathbb{Z}$$
as the theta level\index{theta level}.  Here $H_{\theta}$ is the co-root corresponding to the highest root $\theta$.
\end{definition}

\begin{proposition}\label{KillTheta}\cite{BGMA}
Let $\vec{\lambda} \in \mc{P}_{\ell}(\mathfrak{g})^n$.  Let $J_1,\dots, J_{k+3}$ be any partition of $N$ into $k+3$ non-empty sets. Without loss of generality assume that $\lambda(J_i)=\sum_{a\in J_i} |\lambda_i|$ are ordered, i.e. $\lambda(J_1)\leq \dots \leq \lambda(J_{k+3})$. If $\sum_{i=1}^{k+2}\lambda(J_i) \leq \ell+1$, then $c_k(\mathbb{V}(\mathfrak{g},\vec{\lambda},\ell))$ contracts the $k$-dimensional $F$-cycle $Z_{\vec{J}}$, 
and in particular, $c_k(\mathbb{V}(\mathfrak{g},\vec{\lambda},\ell))$ is extremal in $\op{N}^k(\ovop{M}_{0,n})$.
\end{proposition}

The proof of Proposition \ref{KillTheta} is analogous to that of Proposition \ref{empty}.

\section{Full dimensional subcones of Pliant cones}\label{Basis}

\subsection{Fakhruddin's Basis and the Pliant Cone for $\ovop{M}_{0,n}$}\label{FakhruddinBasis}
\begin{claim}\label{Spanning}There is a spanning set for the Pliant Cone $\operatorname{Pl}^m(\ovop{M}_{0,n})$, given by  a basis of first Chern classes of vector bundles of conformal blocks. 
\end{claim}

\begin{proof} 
By \cite{KeelThesis}, $A^1(\ovop{M}_{0,n})$ 
generates  $A^m(\ovop{M}_{0,n})$, all $m$ \cite{KeelThesis}. 

There is at least one basis we may use for the Picard group of $\ovop{M}_{0,n}$ (see Section \ref{OtherBasis} for more choices in the $\op{S}_n$ invariant case).  Namely, the bundles $\mathcal{B}$ that generate Fakhruddin's basis for $\op{Pic}(\ovop{M}_{0,n})$,  are  
$$\mathcal{B}=\{\mathbb{V}(\sL(2), \vec{\lambda}, 1) : \op{rk}(\mathbb{V}(\sL(2), \vec{\lambda}, 1)  \ne 0\}.$$
In $\mathcal{B}$ bundles are determined by n-tuples of weights of the form $\vec{\lambda}=(\lambda_1,\ldots,\lambda_n)$, where  $\lambda_i \in \{0,\omega_1\}$,  $0\ne \sum_i |\lambda_i|$ is divisible by $2$ and such that 
at least four weights $\lambda_i$ are different than zero.  Moreover, all have level one, and so also have rank one.
We note that if $n$ is odd, then all elements of $\mathcal{B}$ are pulled back from $\ovop{M}_{0,n-1}$ and if $n$ is even, then $\mathbb{V}(\sL(2),\{\omega_1^{n}\}, 1)$ is the unique element of $\mathcal{B}$ that is not pulled back from $\ovop{M}_{0,n-1}$.

\end{proof}
\begin{remark}
The divisors generating the spanning set for $\op{A}^1(\ovop{M}_{0,n})$ are all extremal: In case $n$ is odd, then they are all pulled back from $\ovop{M}_{0,n-1}$ and intersect the fibral curve in the projection map in degree zero.  If $n$ is odd, then except for $E=c_1(\mathbb{V}(\sL(2), \omega_1^n, 1))$, they are all pulled back from $\ovop{M}_{0,n-1}$ and so are extremal.   In fact, by \cite{ags}, $E$ intersects  a number of $\op{F}$-curves in degree zero, and is therefore extremal.
\end{remark}

\begin{remark}We note that there are many identities giving relations between different conformal blocks divisors, and in particular, as the divisors making up Fakhruddin's basis all have rank one, they each satisfy two different scaling identities.  To express them, write any element of Fakhruddin's basis as $c_1(\mathbb{V}(\sL(2), (c_1\omega_1,\ldots, c_n \omega_1), 1))$, where the coefficients $c_i \in \{0,1\}$.
\begin{multline}
c_1(\mathbb{V}(\sL(2), (c_1\omega_1,\ldots, c_n \omega_1), 1))\\=\frac{1}{m}c_1(\mathbb{V}(\sL(2m), (c_1\omega_m,\ldots, c_n \omega_m), 1))
=\frac{1}{m}c_1(\mathbb{V}(\sL(2), (mc_1\omega_1,\ldots, mc_n \omega_1), m)).
\end{multline}
Moreover, at least in the $\op{S}_n$-invariant cases, there are other natural basis given by conformal blocks divisors.  For example, the set of bundles studied in \cite{agss} is known to form a basis for $\op{Pic}(\ovop{M}_{0,n})^{\op{S}_n}$ for $n \le 2000$ (verified using a computer by A. Kazanova), and that studied in \cite{ags}, is known to form a basis for $\op{Pic}(\ovop{M}_{0,n})^{\op{S}_n}$ for all $n$.     These are discussed in Section \ref{OtherBasis}.   
\end{remark}

\begin{remark}Naturally one should ask whether the Pliant or the nef cones $\op{Pl}^m(\ovop{M}_{0,n})$, and $\op{Nef}^m(\ovop{M}_{0,n})$ are generated by classes that come from vector bundles of conformal blocks.  Swinarski showed that conformal blocks divisors for $\sL(2)$ do not cover the whole nef cone of $\ovop{M}_{0,6}$.  However, as Fakhruddin comments in \cite{Fakh},  some computations  suggest that the cone of divisors spanned by conformal blocks for $\sL(2)$ is strictly contained in that generated by divisors of $\sL(3)$.  \end{remark}

\subsection{The $\op{S}_n$-invariant case}\label{OtherBasis}
One may use products of $m$ basis elements for $\operatorname{Pic}(\operatorname{M}_{0,n})^{\operatorname{S}_n}$ to generate  subcones of the Pliant cone 
$\operatorname{Pl}^m(\operatorname{M}_{0,n})^{\operatorname{S}_m}$.  Unlike with Keel's result for $\ovop{M}_{0,n}$, one does not have that 
$\op{A}^m(\ovop{M}_{0,n})^{\op{S}_n}$ is generated by $\op{A}^1(\ovop{M}_{0,n})^{\op{S}_n}$.  For example, while $\op{A}^1(\ovop{M}_{0,7})^{\op{S}_7}$  has 2 generators (see the table below), one can check that the dimension of
$\op{A}^2(\ovop{M}_{0,7})^{\op{S}_7}$ is $4$.

 Two  basis for $\op{Pic}(\op{M}_{0,n})^{\op{S}_n}$ are given by the following sets: If $n=2(g+1)$ is even, then by \cite{ags}
$$\mathcal{B}_1=\{\mathbb{V}(\sL(2), \omega_1^n, \ell ) : 1 \le  \ell \le g\},$$
and if  $n=2(g+1)+1$ is odd, 
$$\mathcal{B}_1=\{\mathbb{V}(\sL(2), \{\omega_1^{n-1},0\}, \ell ) : 1 \le  \ell \le g\}$$
is a basis for $\op{Pic}(\ovop{M}_{0,n})^{\op{S}_n}$.  And at least for $n \le 2,000$, by \cite{agss} and a computer check,
$$\mathcal{B}_2=\{\mathbb{V}(\sL(n), \{\omega_i^n\}, 1 ) : 2 \le  i \le \lfloor \frac{n}{2} \rfloor\}.$$
is a basis for $\op{Pic}(\ovop{M}_{0,n})^{\op{S}_n}$.  The second basis $\mathcal{B}_2$ may be more interesting,  as was shown in \cite{agss} all of its elements span extremal rays of the nef cone $\op{Nef}^1(\ovop{M}_{0,n})^{\op{S}_n}$.  All elements of $\mathcal{B}_2$ are of level one, and hence rank one, and so the only nontrivial classes will be products of first Chern classes.   In contrast, only the first Chern classes of bundles with $\ell \in \{1,2,g-1,g\}$ span extremal rays of the nef cone.  In particular, starting at $n=12$, the $\mc{B}_2$ gives more extremal rays of the nef cone than does $\mc{B}_1$.  Elements of $\mathcal{B}_1$ generally have rank greater than one, so products of first Chern classes will sometimes be equivalent to higher Chern classes: For instance, 
by the Critical Level Identity, $$c_1(\mathbb{V}(\sL(2), \{\omega_1^{2(g+1)}\}, g ))^k=c_k(\mathbb{V}(\sL(2), \{\omega_1^{2(g+1)}\}, g )).$$ We give a few  explicit examples below.

\bigskip
\bigskip

\begin{tabular}{c c c c c  c}

\hline 
n  & $[B_2,B_3]$ &  $\mathbb{V}$  & rank\ &  $\mathbb{V}$  & rank \\
\hline

 & &  &  \\
 $6$ &$[2,1]$  & $\alpha_1=c_1(\mathbb{V}(\sL(2),\{\omega_1^6\},1))$ & 1 & $\alpha_1=c_1(\mathbb{V}(\sL_{6},\{\omega_3^6\},1))$ & 1 \\
$6$ &$[1,3]$  & $\alpha_2=c_1(\mathbb{V}(\sL(2),\{\omega_1^6\},2))$ & 2 &  $\alpha_2=c_1(\mathbb{V}(\sL_{6},\{\omega_2^6\},1))$ & 1  \\
$7$ &$[1,1]$  & $\beta_1=c_1(\mathbb{V}(\sL(2),\{\omega_1^6,0\},1))$ & 1 & $\beta_1=c_1(\mathbb{V}(\sL_{7},\{\omega_3^7\},1))$ &  1 \\
$7$ &$[1,3]$  & $\beta_2=c_1(\mathbb{V}(\sL(2),\{\omega_1^6,0\},2))$ & 4 & $\beta_2=c_1(\mathbb{V}(\sL_{7},\{\omega_2^7\},1))$ &  1  \\
$8$ &$[3,2,4]$  & $\gamma_1=c_1(\mathbb{V}(\sL(2),\{\omega_1^8\},1))$  &  1 & $\gamma_1=c_1(\mathbb{V}(\sL_{8},\{\omega_4^8\},1))$  &   1 \\
$8$ &$[6,11,8]$  & $\gamma_2=c_1(\mathbb{V}(\sL(2),\{\omega_1^8\},2))$ & 8  & $\gamma_2=c_1(\mathbb{V}(\sL_{8},\{\omega_3^8\},1))$ &  1  \\
$8$ &$[1,3,6]$  & $\gamma_3=c_1(\mathbb{V}(\sL(2),\{\omega_1^8\},3))$  &   13 & $\gamma_3=c_1(\mathbb{V}(\sL_{8},\{\omega_2^8\},1))$  &  1\\
\end{tabular}

\bigskip
In the table above, $B_2=\sum_{ij}\delta_{ij}$, and $B_3=\sum_{ijk}\delta_{ijk}$.
\bigskip
\bigskip

Elements in $\op{Pl}^2(\ovop{M}_{0,7})^{\op{S}_7}$ will be given by products of extremal rays of the cone of $\op{S}_7$-invariant nef divisors:
$$(\pi_7^*\alpha_1)^2=\beta_1^2, \ \ (\pi_7^*\alpha_1)\cdot (\pi_7^*\alpha_2)=\beta_1 \cdot \beta_2, \ \  \mbox{ and } \  (\pi_7^*\alpha_2)^2=\beta_2^2.$$
Moreover: $c_1(\mathbb{V}(\sL(2),\{\omega_1^6,0\},2))^2=c_2(\mathbb{V}(\sL(2),\{\omega_1^6,0\},2))$, by the Critical Level identity.  Elements in $\op{Pl}^2(\ovop{M}_{0,8})^{\op{S}_8}$ will be given by $\gamma_1^2$, $ \gamma_2^2$, $\gamma_3^2$, $\gamma_1\cdot \gamma_2$, $\gamma_1\cdot \gamma_3$, and $\gamma_2\cdot \gamma_3$.  
Again  by the Critical Level identity, $c_1(\mathbb{V}(\sL(2),\{\omega_1^8\},2))^2=c_2(\mathbb{V}(\sL(2),\{\omega_1^8\},2))$.

\section{Some questions}
\subsection{Powers of first Chern classes}
As with the generators of  these subcones of  $\op{Pl}^m(\ovop{M}_{0,n})$,  we have seen in many examples, that  $m$-th Chern classes are often a product of first Chern classes.   This is not expected:
 If $\mathbb{V}$ is a vector bundle with a projectively flat connection on a smooth projective variety $X$ with trivial fundamental group, then  $c_m(\mathbb{V})=\frac{1}{R^m}{R \choose m} \ c_1(\mathbb{V})^m$, where $R=\op{rk}(\mathbb{V})$.  
The spaces  $\ovop{M}_{0,n}$ are simply connected \cite{BP}, and vector bundles of conformal blocks carry a projectively flat connection on the interior $\op{M}_{0,n}$ \cite{TUY}.  The connection for a  vector bundle of conformal blocks on $\op{M}_{0,n}$ does not extend to a projectively flat connection on the boundary.  It would be interesting to know which bundles have this property, and what it might mean geometrically.

\subsection{Quasi-polynomiality}
 In \cite{BG}, Belkale and Gibney show that  The Chern classes of $\mathbb{V}(\sL(r),\ell)$ on $\ovmc{M}_g$ are quasi-polynomial in l for sufficiently large $\ell$.  It would be interesting to know if the same is true on $\ovop{M}_{0,n}$.

\subsection{Nonvanishing} As was first discovered by Fakhruddin, there are examples of bundles of positive rank, whose first Chern classes are unexpectedly zero.     The classes $c_m(\mathbb{V})$ are positive in the sense that they nonnegatively intersect all effective cycles of dimension $m$ on $\ovop{M}_{0,n}$.  However, as can be seen in the expression for these classes given in Theorem \ref{Chern}, these are alternating sums  of  products of $\psi_i$ and boundary classes, and it is not clear at all when they are going to be nonzero.  It is natural to ask, as was done in \cite{BGMB} for first Chern classes, whether necessary and sufficient conditions can be given so that the $c_m(\mathbb{V})$ is nonzero\footnote{Such criteria exist for first Chern classes of vector bundles of conformal blocks for $\sL(2)$ in \cite{BGMB}.}.  

\subsection{Acknowledgements}
We thank P. Belkale, who was our coauthor on  the results that  here we have formally extended to higher Chern classes. We also thank Izzet Coskun, Mihai Fulger, Brian Lehmann,  Alina Marian, and John Ottem for useful discussions and clarifications.  AG was supported by  NSF  DMS-1601909 and DMS-1344994. SM was supported in part by a Simons Travel Grant and by NSF DMS-1361159 (PI: Patrick Brosnan).

\vspace{0.1 in}
\noindent
A.G.: Department of Mathematics, University of Georgia, Athens, GA 30602\\
email: agibney@math.uga.edu

\vspace{0.2 cm}

\noindent
S.M.: Department of Mathematics, University of Maryland,
College Park, MD 20742\\
email: swarnava@umd.edu


\begin{bibdiv}
\begin{biblist}

\bib{ags}{article}{
   author={Alexeev, Valery},
   author={Gibney, Angela},
   author={Swinarski, David},
   title={Higher-level $\germ{sl}_2$ conformal blocks divisors on
   $\overline M_{0,n}$},
   journal={Proc. Edinb. Math. Soc. (2)},
   volume={57},
   date={2014},
   number={1},
   pages={7--30}
  }

\bib{agss}{article}{
  author={Arap, Maxim},
  author={Gibney, Angela},
  author={Stankewicz, Jim},
  author={Swinarski, David},
  title={$\sL _n$ level $1$ Conformal blocks divisors on $\overline {\operatorname {M}}_{0,n}$},
  journal={International Math Research Notices},
  date={2011}
  }
\bib{BF}{article}{
  author={Belkale, Prakash},
  author={Fakhruddin, Najmuddin},
  title={Triviality properties of principal bundles on singular curves},
  date={2016},
  note={arXiv:1509.06425v3 [math.AG]},
}

\bib{Beauville}{article}{
  author={Beauville, Arnaud},
  title={Conformal blocks, fusion rules and the Verlinde formula},
  conference={ title={}, address={Ramat Gan}, date={1993}, },
  book={ series={Israel Math. Conf. Proc.}, volume={9}, publisher={Bar-Ilan Univ.}, place={Ramat Gan}, },
  date={1996},
  pages={75--96}
}

	
	

\bib{BGMA}{article}{
   author={Belkale, Prakash},
   author={Gibney, Angela},
   author={Mukhopadhyay, Swarnava},
   title={Vanishing and identities of conformal blocks divisors},
   journal={Algebr. Geom.},
   volume={2},
   date={2015},
   number={1},
   pages={62--90}
}
		
\bib{BGMB}{article}{
   author={Belkale, P.},
   author={Gibney, A.},
   author={Mukhopadhyay, S.},
   title={Nonvanishing of conformal blocks divisors on $\overline M_{0,n}$},
   journal={Transform. Groups},
   volume={21},
   date={2016},
   number={2},
   pages={329--353}
}

\bib{BG}{article}{
  author={Belkale, Prakash},
  author={Gibney, Angela},
  title={On finite generation of the section ring of the determinant of cohomology line bundle},
  journal={arXiv:1606.08726},
  date={2016},
}
\bib{BP}{article}{
  author={Boggi, M},
  author={Pikaart, M.},
  title={Galois covers of moduli of curves},
  journal={Comp. Math.},
  volume={120},
  date={2000},
  pages={171-191}
}



		



\bib{CLO}{article}{
   author={Coskun, Izzet},
   author={Lesieutre, John},
   author={Ottem, John Christian},
   title={Effective cones of cycles on blow-ups of projective space},
   journal={arXiv:1603.04808},
   date={2016},
   pages={1--26}
}



\bib{DELV}{article}{
   author={Debarre, Olivier},
   author={Ein, Lawrence},
   author={Lazarsfeld, Robert},
   author={Voisin, Claire},
   title={Pseudoeffective and nef classes on abelian varieties},
   journal={Compos. Math.},
   volume={147},
   date={2011},
   number={6},
   pages={1793--1818}
}
		

\bib{EV}{article}{
author={Esnault, Helene},
  author={Viehweg, Eckhart},
  title={Logarithmic De-Rham complexes},
  journal={Inventiones Mathematica},
  date={1980}
}


\bib{Fakh}{article}{
   author={Fakhruddin, Najmuddin},
   title={Chern classes of conformal blocks},
   conference={
      title={Compact moduli spaces and vector bundles},
   },
   book={
      series={Cont. Math.},
      volume={564},
      publisher={Amer. Math. Soc.},
      place={Providence, RI},
   },
   date={2012},
   pages={145--176}
}

\bib{LF}{article}{
  author={Fulger, Mihai},
  author={Lehmann, Brian},
  title={Positive cones of dual cycle classes},
  journal={arXiv:1408.5154 [math.AG]},
  date={2014}
}



\bib{FSV2}{article} {
    author = {Feigin, B. L.},
    author= {Schechtman, V. V.},
    author={Varchenko, A. N.},
     TITLE= {On algebraic equations satisfied by correlators in
              {W}ess-{Z}umino-{W}itten models},
   JOURNAL = {Lett. Math. Phys.},
    VOLUME = {20},
      YEAR = {1990},
    NUMBER = {4},
     PAGES = {291--297},
       URL = {http://dx.doi.org/10.1007/BF00626525},
}



\bib{Kaz}{article}{
   author={Kazanova, Anna},
   title={On $S_n$-invariant conformal blocks vector bundles of rank one
   on $\overline{M}_{0,n}$},
   journal={Manuscripta Math.},
   volume={149},
   date={2016},
   number={1-2},
   pages={107--115}
}

\bib{KeelThesis}{article}{
   author={Keel, Sean},
   title={Intersection theory of moduli space of stable $n$-pointed curves
   of genus zero},
   journal={Trans. Amer. Math. Soc.},
   volume={330},
   date={1992},
   number={2},
   pages={545--574}
}
	




\bib{LO}{article}{
   author={Lesieutre, John},
   author={Ottem, John Christian},
   title={Curves disjoint from a nef divisor},
   journal={Michigan Math. J.},
   volume={65},
   date={2016},
   number={2},
   pages={321--332}
}

\bib{M2}{article}{
          author = {Grayson, Daniel R.},
          author={Stillman, Michael E.},
          title = {Macaulay2, a software system for research
                   in algebraic geometry},
         note= {Available at
           \texttt{http://www.math.uiuc.edu/Macaulay2/}}
        }


\bib{Muk3}{article}{
    AUTHOR = {Mukhopadhyay, Swarnava},
     TITLE = {Rank-level duality and conformal block divisors},
   JOURNAL = {Adv. Math.},
    VOLUME = {287},
      YEAR = {2016},
     PAGES = {389--411}
}




\bib{MOPPZ}{article}{
  author={Marian, A.},
  author={Oprea, D.},
  author={Pandharipande, R.},
  author={Pixton, A.},
  author={Zvonkine, D.},
  title={The Chern character of the Verlinde bundle over the moduli space of stable curves},
  journal={arXiv:1311.3028 [math.AG]},
  date={2014},
  note={to appear in Journal f\"ur die reine und angewandte Mathematik},
}



\bib{Mead}{article}{
   author={Mead, D. G.},
   title={Newton's identities},
   journal={Amer. Math. Monthly},
   volume={99},
   date={1992},
   number={8},
   pages={749--751}}
   
		
		
\bib{Ottem}{article}{
   author={Ottem, John},
   title={Nef cycles on some hyperk\"ahler fourfolds},
    journal={arXiv:1505.01477},
   date={2015},
   pages={1--12}
   }
   
   
	
\bib{Schur}{article}{
		author={Schenck, Hal},
		author={Stillman, Michael},
		title={\texttt{\upshape SchurRings}: a Macaulay2 package make computations in the representation ring of GL(n) possible},
		date={2007},
		note={Version 0.2, {http://www.math.uiuc.edu/Macaulay2/}},
}
	\bib{ConfBlocks}{article}{
		author={Swinarski, David},
		title={\texttt{\upshape ConformalBlocks}: a Macaulay2 package for computing conformal block divisors},
		date={2010},
		note={Version 1.1, {http://www.math.uiuc.edu/Macaulay2/}},
}

\bib{Sorger}{article}{
   author={Sorger, Christoph},
   title={La formule de Verlinde},
   language={French, with French summary},
   note={S\'eminaire Bourbaki, Vol.\ 1994/95},
   journal={Ast\'erisque},
   number={237},
   date={1996},
   pages={Exp.\ No.\ 794, 3, 87--114},
   issn={0303-1179},
   review={\MR{1423621}},
}
		

\bib{TUY}{article}{
  author={Tsuchiya, Akihiro},
  author={Ueno, Kenji},
  author={Yamada, Yasuhiko},
  title={Conformal field theory on universal family of stable curves with gauge symmetries},
  conference={ title={Integrable systems in quantum field theory and statistical mechanics}, },
  book={ series={Adv. Stud. Pure Math.}, volume={19}, publisher={Academic Press}, place={Boston, MA}, },
  date={1989},
  pages={459--566}
}

\bib{Tsu}{article}{
    AUTHOR = {Tsuchimoto, Yoshifumi},
     TITLE = {On the coordinate-free description of the conformal blocks},
   JOURNAL = {J. Math. Kyoto Univ.},
    VOLUME = {33},
      YEAR = {1993},
    NUMBER = {1},
     PAGES = {29--49}
}



\bib{VoisinNef}{article}{
  author={Voisin, Claire},
  title={Coniveau 2 complete intersections and effective cones},
  journal={Geom. Funct. Anal.},
  volume={19},
  date={2010},
  number={5},
  pages={1494--1513}
}


\end{biblist}
\end{bibdiv}
\end{document}